\documentclass[12pt, reqno]{amsart}
\pdfoutput=1
\usepackage{amssymb}
\usepackage{amsfonts}
\usepackage{amsmath}
\usepackage{amsthm}
\usepackage{color}
\usepackage{hyperref}
\usepackage{graphicx}
\usepackage{amsaddr}

\newtheorem{theorem}{Theorem}[section]
\newtheorem{proposition}[theorem]{Proposition}
\newtheorem{definition}[theorem]{Definition}
\newtheorem{lemma}[theorem]{Lemma}
\newtheorem{corollary}[theorem]{Corollary}
\newtheorem{example}[theorem]{Example}

\def\C{\mathcal{C}}
\def\Z{\mathbb{Z}}
\def\F{\mathbb{F}}
\def\P{\mathcal{P}}

\def\Aut{\mathop{\mathrm{Aut}}}
\def\dev{\mathop{\mathrm{dev}}}

\begin{document}

\title{Strongly regular configurations}

\author[M.~Abreu, M.~Funk, V.~Kr\v{c}adinac, and D.~Labbate]{Mari\'{e}n Abreu$^1$, Martin Funk$^1$,
Vedran Kr\v{c}adinac$^2$, and Domenico Labbate$^1$}

\address{$^1$Dipartimento di Matematica, Informatica ed Economia,
Universit\`{a} degli Studi della Basilicata,
Viale dell'Ateneo Lucano~10, Potenza, Italy}

\address{$^2$Department of Mathematics, Faculty of Science, University of Zagreb,
Bijeni\v{c}ka cesta~30, HR-10000 Zagreb, Croatia}

\email{marien.abreu@unibas.it}
\email{martin.funk@unibas.it}
\email{vedran.krcadinac@math.hr}
\email{domenico.labbate@unibas.it}

\thanks{The third author was supported by
the Croatian Science Foundation under the projects
$6732$ and $9752$.}



\date{September 12, 2021}

\maketitle

\begin{abstract}
We study combinatorial configurations with the associated point and
line graphs being strongly regular. Examples not belonging to known
classes such as partial geometries and their generalizations or elliptic
semiplanes are constructed. Necessary existence conditions are proved and
a table of feasible parameters of such configurations with at most $200$
points is presented. Non-existence of some configurations with feasible
parameters is proved.
\end{abstract}

\section{Introduction}

A \emph{(combinatorial) configuration} is a finite partial linear space with constant point and line degrees.
If there are~$v$ points of degree~$r$ and~$b$ lines of degree~$k$, the parameters are written
$(v_r,b_k)$. If $v=b$, or equivalently $r=k$, the configuration is called \emph{symmetric} and
the parameters are written~$(v_k)$. Throughout the paper we assume $k\ge 3$ and $r\ge 3$.

The \emph{point graph} of a configuration has the $v$ points as vertices, with
two vertices being adjacent if the points are collinear.
The \emph{line graph} is defined dually: the $b$ lines of the configuration are the vertices,
and adjacency is concurrence.\footnote{We will always use the term \emph{line graph} in this sense,
and not in the sense of graph theory (the line graph $L(G)$ of a graph $G$).}
The point and line graphs are regular of degrees $r(k-1)$ and $k(r-1)$, respectively.
A graph is called \emph{strongly regular} with parameters
$SRG(n,d,\lambda,\mu)$ if it has~$n$ vertices, is regular of degree~$d$, and
every two vertices have $\lambda$ common neighbors if
they are adjacent and $\mu$ common neighbors if they are not
adjacent. We are interested in configurations with both associated graphs
being strongly regular.

A prominent family of such configurations are the \emph{partial
geometries} $pg(s,t,\alpha)$, introduced by
R.~C.~Bose~\cite{RCB63}. These are configurations with $r=t+1$ and
$k=s+1$ such that for every non-incident point-line pair $(P,\ell)$,
there are exactly~$\alpha$ points on $\ell$ collinear with $P$. The
point graph is a
\begin{equation}\label{pgpointgraph}
SRG\left(\frac{(s+1)(st+\alpha)}{\alpha},\, s(t+1),\, s-1+t(\alpha-1),\,\alpha(t+1)\right),
\end{equation}
and the line graph is a
\begin{equation}\label{pglinegraph}
SRG\left(\frac{(t+1)(st+\alpha)}{\alpha},\, t(s+1),\, t-1+s(\alpha-1),\, \alpha(s+1)\right).
\end{equation}
Partial geometries include Steiner $2$-designs
$pg(s,t,s+1)$ and their duals $pg(s,t,t+1)$, Bruck nets $pg(s,t,t)$~\cite{RHB51, RHB63}
and their duals $pg(s,t,s)$ (transversal designs), %of index one),
and generalized quadrangles $pg(s,t,1)$ as special cases.

If $v\neq b$, partial geometries are the only configurations with both associated graphs strongly regular.
This follows from \cite[Theorem 1.2]{BHT97}:

\begin{theorem}\label{thmBHT}%[Brouwer-Haemers-Tonchev]
Let the point graph of a $(v_r,b_k)$ configuration be strongly regular. Then the configuration is a
partial geometry or $v\le b$. If $v=b$, then $\det(A+k I)$ is a square, where $A$ is
the adjacency matrix of the point graph.
\end{theorem}

If $v=b$, there are such configurations that are not partial geometries. The smallest
examples are $(10_3)$ configurations with associated graphs $SRG(10,6,3,4)$ (the
complement of the Petersen graph). One such configuration is
the Desargues configuration, which is a semipartial geometry for $\alpha=2$ and
$\mu=4$ (see Section~\ref{sec3} for the definition). There is another
such configuration not belonging to the known generalizations of partial geometries
such as semipartial geometries~\cite{DT78,DV95,FDC03} and strongly regular
$(\alpha,\beta)$-geometries~\cite{HM01}, represented in Figure~\ref{smallest}.

\begin{figure}[t]
\begin{center}
\includegraphics[width=6cm]{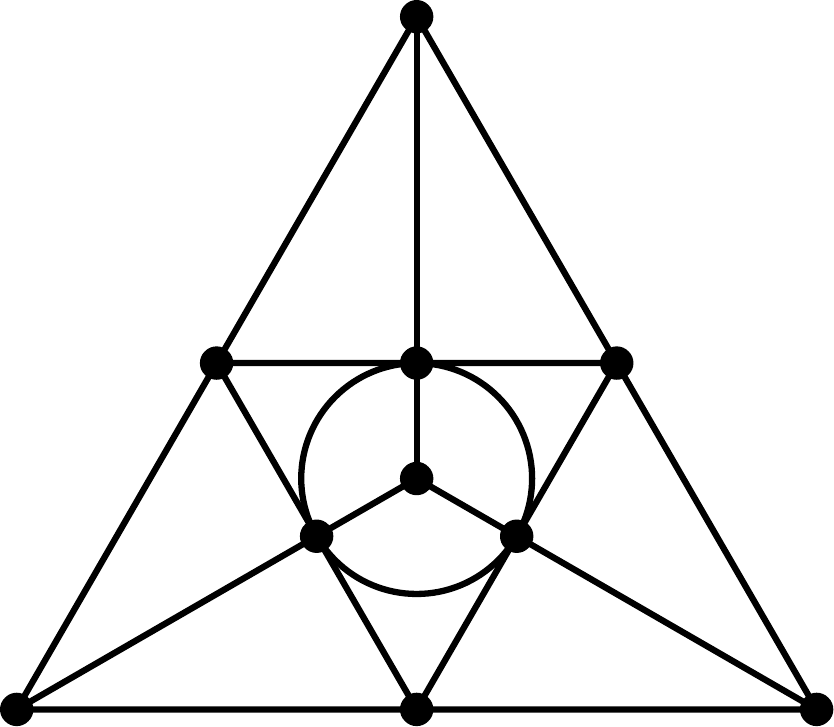}
\end{center}
\caption{A strongly regular configuration that is not an $(\alpha,\beta)$-geometry.}\label{smallest}
\end{figure}

In this paper we study combinatorial configurations similar to this one.
In Section~\ref{sec2} we give the definition of a strongly regular configuration.
The concept unifies known classes such as (semi)partial geometries and elliptic
semiplanes~\cite{PD68} with several sporadic examples from the literature (see
Remark on page 37 of~\cite{BHT97} and Section 7.2 of~\cite{BKK03}). We focus on
strongly regular configurations that are proper and primitive, not belonging to
the known classes. We prove two necessary conditions on the parameters of strongly
regular configurations, stronger than conditions on the parameters of the associated
strongly regular graphs.

In Section~\ref{sec3} we present families of strongly regular configurations.
A family associated with Moore graphs and a family constructed from the points
and planes of the finite projective space $PG(4,q)$ have the same parameters as
semipartial geometries. We prove that there are strongly regular configurations
with these parameters that are not semipartial geometries. A third family constructed
from finite projective planes has parameters not compatible with semipartial geometries.

Section~\ref{sec4} contains constructions of strongly regular configurations
from difference sets in groups. Some of the configurations from the previous
section can also be constructed in this way. We perform an exhaustive search
in groups of order $v\le 200$ and find three other parameter sets $(v_k;\lambda,\mu)$
for which strongly regular configurations exist. Configurations with a fourth
parameter set are constructed in a different manner.

In the final Section~\ref{sec5} we present a table of feasible parameters of
strongly regular configurations with $v\le 200$. An on-line version of the
table with links to the actual configurations is available on the web page
\begin{center}
\url{https://web.math.pmf.unizg.hr/~krcko/results/srconf.html}
\end{center}
We perform complete classifications of configurations with small parameters
and prove non-existence for infinitely many feasible
parameter sets corresponding to rook graphs.

\section{Definitions and conditions on the parameters}\label{sec2}

In view of the motivation presented in the Introduction, we make the following definition.

\begin{definition}\label{maindef}
A symmetric configuration will be called a \emph{strongly regular configuration}
with parameters $(v_k;\lambda,\mu)$ if the associated point graph is a
$SRG(v,k(k-1),\lambda,\mu)$.
\end{definition}

We can prove that the line graph is also strongly regular with the same para\-meters.
We use the following lemma from~\cite{AMC81}.

\begin{lemma}\label{cohenlemma}
Suppose that the point graph of a $(v_r,b_k)$ configuration is strongly regular
with parameters~\eqref{pgpointgraph} corresponding to a partial geometry $pg(s,t,\alpha)$.
Then the configuration is a $pg(s,t,\alpha)$.
\end{lemma}

\begin{theorem}\label{linegraph}
Given a strongly regular $(v_k;\lambda,\mu)$ configuration, the associated line graph
is a $SRG(v,k(k-1),\lambda,\mu)$.
\end{theorem}

\begin{proof}
A graph is strongly regular with parameters $SRG(v,d,\lambda,\mu)$ if and only if its adjacency matrix~$A$
satisfies
\begin{equation}\label{adjsrg}
A^2= dI+\lambda A+\mu (J-I-A).
\end{equation}
Here $I$ and $J$ are the $v\times v$ identity matrix and the
all-one matrix. Let~$N$ be the incidence matrix of the configuration. Then, $A=N N^t-k I$ and $B=N^t N-k I$
are adjacency matrices of the point and line graphs, respectively.
By~\eqref{adjsrg}, we have
\begin{equation*}
N N^t N N^t +(\mu-\lambda-2k)N N^t+(k(\lambda-\mu+1)+\mu)I - \mu
J=0.
\end{equation*}
If the incidence matrix~$N$ is non-singular, we can multiply by $N^{-1}$ from the left
and by $N$ from the right. Using $N^{-1} J = \frac{1}{k} J$ and $J N=k J$, we get
\begin{equation*}
N^t N N^t N + (\mu-\lambda-2k)N^t N +(k(\lambda-\mu+1)+\mu)I - \mu
J=0.
\end{equation*}
This is equation~\eqref{adjsrg} for the matrix~$B$, and therefore the line graph is
also strongly regular with the same parameters.

Now assume that the incidence matrix $N$ is singular. Then the matrix~$N$ has eigenvalue~$0$
and the matrix~$A$ has eigenvalue $-k$. Thus, $k^2-(\mu-\lambda)k+\mu-k(k-1)=0$ holds
and~$k$ divides~$\mu$. Denoting $\alpha=\mu/k$, we see that the parameters of the point graph
correspond to a partial geometry~\eqref{pgpointgraph}. By Lemma~\ref{cohenlemma}, the
configuration is a partial geometry and the line graph is also strongly regular with
parameters~\eqref{pglinegraph}. The same argument was used in the proof of
\cite[Theorem 1.2]{BHT97}.
\end{proof}

We shall call strongly regular configurations with non-singular incidence matrices \emph{proper}.
The previous proof shows that improper configurations must be partial
geometries. The parameters of a strongly regular configuration are not independent. A
necessary condition for the existence of a $SRG(v,k(k-1),\lambda,\mu)$ is
\begin{equation*}
(v-1-k(k-1))\mu=k(k-1)(k(k-1)-1-\lambda).
\end{equation*}
From this, $v$ can be expressed from $k$, $\lambda$, and~$\mu$, provided $\mu\neq 0$.
There are many other necessary conditions on the parameters
of a $SRG(v,k(k-1),\lambda,\mu)$. The adjacency matrix has eigenvalue $k(k-1)$ with
multiplicity~$1$ and two more eigenvalues
\begin{equation}\label{eigenv}
r,s=\frac{1}{2}\left(\lambda-\mu\pm \sqrt{(\lambda - \mu)^2-4(\mu-k(k-1))} \right)
\end{equation}
with respective multiplicities
\begin{equation}\label{multi}
f,g = \frac{1}{2} \left(v-1\mp \frac{(r+s)(v-1)+2k(k-1)}{r-s}\right).
\end{equation}
The multiplicities are integers, giving divisibility conditions on the para\-meters. If $f\neq g$, the
eigenvalues $r$, $s$ are also integers. See~\cite{BvM21} for further necessary conditions
and~\cite{AEB} for tables of para\-meters of strongly regular graphs with up-to-date
information on their existence. The parameters $(v_k;\lambda,\mu)$ of a strongly regular
configuration will be considered \emph{feasible} if the associated strongly regular graphs
exist or their existence cannot be ruled out. On top of that, we assume two more necessary
conditions on the parameters. The first condition follows from \cite[Theorem 1.2]{BHT97}:
$\det(A+kI)=\det(N N^t)=(\det N)^2$ is a square. The matrix $A+kI$ has
eigenvalues $k^2$, $r+k$, $s+k$ with multiplicities $1$, $f$, $g$ and the determinant can be
computed from the parameters.

\begin{proposition}\label{squarecond}
If a strongly regular $(v_k;\lambda,\mu)$ configuration exists, then $(r+k)^f (s+k)^g$ is the square
of an integer, where $r$, $s$, $f$, $g$ are given by \eqref{eigenv} and \eqref{multi}.
\end{proposition}

For example, the condition rules out strongly regular $(28_4;6,4)$ configurations, although
$SRG(28,12,6,4)$ graphs exist. Equations \eqref{eigenv} and \eqref{multi} give $r=4$, $s=-2$,
$f=7$, $g=20$, and $2^{41}$ is not a square. The second condition follows from a counting argument.

\begin{theorem}\label{cliquecond}
If a strongly regular $(v_k;\lambda,\mu)$ configuration exists, then
$(v-k)(\lambda+1)\ge k(k-1)^3$. Equality holds if and only if the
configuration is a partial geometry.
\end{theorem}

\begin{proof}
Fix a line $\ell$ and for any point $P$ not on
$\ell$, denote by~$\alpha_P$ the number of lines through $P$ concurrent with~$\ell$.
Count the number of flags $(P,\ell_1)$ with $\ell_1$ concurrent with $\ell$ in
two ways to obtain
$$\sum \alpha_P = k(k-1)^2.$$
Similarly, counting triples $(P,\ell_1,\ell_2)$, where $\ell_1\neq\ell_2$
are lines through~$P$ concurrent with $\ell$, gives
$$\sum \alpha_P(\alpha_P-1) = k(k-1)(\lambda-(k-2)).$$
The sums are taken over all $P\not\in \ell$. The average
$\alpha_P$ is $\alpha= \frac{k(k-1)^2}{v-k}$. Now we can compute
\begin{align*}
0 &\le \sum (\alpha_P-\alpha)^2 =  \sum \alpha_P(\alpha_P-1) +
(1-2\alpha) \sum \alpha_P + (v-k)
\alpha^2 = \\
 & = k(k-1)(\lambda-k+2) + \left(1-2\,\frac{k(k-1)^2}{v-k}\right)k(k-1)^2 +
 \frac{k^2(k-1)^4}{v-k} = \\
 & = k(k-1)\left( \lambda+1 - \frac{k(k-1)^3}{v-k}\right).
\end{align*}
From this we see that $(v-k)(\lambda+1)\ge k(k-1)^3$ holds, with equality
if and only if $\alpha_P=\alpha$ for all $P\not\in \ell$, i.e.\
the configuration is a partial geometry.
\end{proof}

For example, the parameters $(81_5;1,6)$ do not satisfy
Theorem~\ref{cliquecond} and strongly regular configurations with these parameters
don't exist, although a $SRG(81,20,1,6)$ graph does. An equivalent form
of the inequality $k(\mu-\lambda-1)\le \mu$ follows from Hoffman's bound
on the size of cliques in strongly regular graphs; see \cite[Section
1.3]{BCN89}. Theorem~\ref{cliquecond} characterizes proper
strongly regular configurations by their parameters.

\begin{corollary}
A strongly regular $(v_k;\lambda,\mu)$ configuration that is not a projective plane is proper if and only if
$(v-k)(\lambda+1) > k(k-1)^3$.
\end{corollary}

Projective planes of order~$n$ are partial geometries $pg(n,n,n+1)$ and satisfy
Theorem~\ref{cliquecond} with equality, but have non-singular incidence matrices. The associated
point and line graphs are complete. More generally, we now consider the
case when the associated graphs are imprimitive, i.e.\ $\mu=0$ or $\mu=k(k-1)$ holds. In
the first case the graphs are disjoint unions of complete graphs $m\cdot K_{n^2+n+1}$
and the configuration is a disjoint union of $m$ projective planes of order~$n$. This case
can be characterized as strongly regular configurations with collinearity of points being
an equivalence relation.

The second imprimitive case $\mu=k(k-1)$ is complementary: non-collinearity of points
is an equivalence relation and the associated graphs are complete multipartite. Strongly
regular configurations with these properties are known as \emph{elliptic semiplanes}.
Dembowski~\cite{PD68} defined a \emph{finite semiplane} as a partial
linear space with parallelism of lines and non-collinearity of
points being equivalence relations. A semiplane is of \emph{order} $n$
if the largest degree of a point or line is $n+1$. Dembowski proved
that the set of all degrees is either $\{n-1,n,n+1\}$, $\{n,n+1\}$,
or $\{n+1\}$, and called semiplanes \emph{hyperbolic},
\emph{parabolic}, or \emph{elliptic} accordingly.
%Hyperbolic and parabolic semiplanes are not configurations
Elliptic semiplanes are precisely the strongly regular configurations with $\mu=k(k-1)$.
Most known elliptic semiplanes are of the form $\mathcal{P}-B$, where $\P$ is a
projective plane of order $n$, and $B$ is a closed Baer subset. The
only known exceptions are Baker's semiplane~\cite{RDB78} with parameters
$(45_7;39,42)$ and Mathon's semiplane~\cite{RM07} with
parameters $(135_{12};129,132)$.

In the sequel we focus on strongly regular configurations that are
proper and \emph{primitive}, i.e.\ such that neither collinearity nor
non-col\-li\-ne\-arity of points are equivalence relations. This is equivalent
with $0< \mu < k(k-1)$. Table~\ref{fptable} contains all feasible parameters
of such configurations with $v\le 200$. We first present constructions of proper
and primitive strongly regular configurations.

\section{Families of strongly regular configurations}\label{sec3}

An \emph{$(\alpha,\beta)$-geometry}~\cite{DV94} is a $(v_r,b_k)$ configuration such that
for every non-incident point-line pair $(P,\ell)$, there are
either $\alpha$ or $\beta$ points on~$\ell$ collinear with~$P$.
Thus, a partial geometry is an $(\alpha,\beta)$-geometry with $\alpha=\beta$.
If $\alpha\neq \beta$, the point graph is not necessarily
strongly regular. Geometries with this additional property are called
\emph{strongly regular $(\alpha,\beta)$-geometries} and are studied in~\cite{HM01}.
An important special case are the \emph{semipartial geometries}, introduced in~\cite{DT78}.
They are $(0,\alpha)$-geometries such that for every pair of non-collinear points, there
are exactly~$\mu$ points collinear with both. The parameters are written $(s,t,\alpha,\mu)$,
where $r=t+1$ and $k=s+1$ are the point and line degrees, and the point graph is a
$$SRG\left(1+\frac{s(t+1)(\mu+t(s+1-\alpha)}{\mu},\,s(t+1),\,s-1+t(\alpha-1),\,\mu\right).$$

Strongly regular $(\alpha,\beta)$-geometries with $v=b$ are strongly regular configurations
by Definition~\ref{maindef}. Our introductory example in Figure~\ref{smallest} is
not an $(\alpha,\beta)$-geometry, although the parameters % $(10_3;3,4)$
correspond to a semipartial geometry. If $\ell$ is the line represented as a circle, there are points~$P$ ouside~$\ell$
with $1$, $2$, or $3$ points on~$\ell$ collinear with~$P$. This example is part of a family associated with Moore
graphs of diameter two, i.e.\ strongly regular graphs with $\lambda=0$ and $\mu=1$.

Moore graphs have parameters $SRG(k^2+1,k,0,1)$ with
$k\in \{2,3,7,$ $57\}$ \cite{HS60}. There is a unique graph for $k=2$ (the pentagon),
$k=3$ (the Petersen graph), and $k=7$ (the Hoffman-Sigleton graph), while
for $k=57$ the existence of such a graph is unknown. The incidence structure with points
being vertices of a $SRG(k^2+1,k,0,1)$ and lines being neighborhoods of single vertices
is a semipartial geo\-metry with $s=t=\alpha=k-1$ and $\mu=(k-1)^2$~\cite{DT78}. The point
graph is the complementary $SRG(k^2+1,k(k-1),k(k-2),(k-1)^2)$. Hence, this incidence
structure is a strongly regular $(v_k;\lambda,\mu)$ configuration
with $v=k^2+1$, $\lambda=k(k-2)$, and $\mu=(k-1)^2$.

For $k=3$, the semipartial geometry
is the Desargues configuration and there is one other $(10_3;3,4)$ configuration given in
Figure~\ref{smallest}. For $k=7$, the semipartial geometry has
full automorphism group $PSU(3,5) : \Z_2$ of order $252000$ acting flag-transitively.
We found $210$ other $(50_7;35,36)$ configurations that are not $(\alpha,\beta)$-geometries.
The semipartial geometry and $110$ of the new examples are self-dual and the remaining ones
form $50$ dual pairs.

\begin{proposition}\label{src50}
There are at least $211$ non-isomorphic $(50_7;35,36)$
configurations, one of which is a semipartial geometry.
Orders of their full automorphism
groups are given in Table~\ref{table50}.
\end{proposition}

\begin{table}[h]
\begin{tabular}{|c c|c c|c c|c c|c c|}
\hline
$|\Aut|$ & \#Cf & $|\Aut|$ & \#Cf & $|\Aut|$ & \#Cf & $|\Aut|$ & \#Cf & $|\Aut|$ & \#Cf \\
\hline
252000 & 1 & 120 & 1 & 40 & 1 & 20 & 6 & 6 & 13 \\
2520 & 1 & 96 & 1 & 36 & 1 & 16 & 3 & 4 & 15 \\
1440 & 1 & 72 & 1 & 32 & 1 & 12 & 1 & 3 & 18 \\
720 & 1 & 48 & 1 & 24 & 6 & 10 & 1 & 2 & 46 \\
240 & 1 & 42 & 1 & 21 & 2 & 8 & 11 & 1 & 76 \\
\hline
\end{tabular}
\vskip 3mm \caption{Distribution of $(50_7;35,36)$ configurations by
order of full automorphism group.}\label{table50}
\end{table}

The configurations of Proposition~\ref{src50} are
available through the on-line version of Table~\ref{fptable}. They were
constructed computationally, by prescribing automorphism groups and
switching submatrices of the incidence matrix:
$$\left[\begin{array}{c c} 1 & 0\\ 0 & 1\\ \end{array}\right] \longleftrightarrow
\left[\begin{array}{c c} 0 & 1\\ 1 & 0\\ \end{array}\right].$$
We used GAP~\cite{GAP4} and our own programs written in C. To check for
isomorphism and compute full automorphism groups, we used nauty~\cite{MP14}.
The construction method for configurations with prescribed automorphism groups
is similar to constructions of quasi-symmetric designs in~\cite{KVK20, VK20} and
relies on the clique-finding program Cliquer~\cite{NO03}.

Another family of semipartial geometries is family (g) from~\cite{DT78}, denoted by $LP(n,q)$ in~\cite{DV95,FDC03}.
The points of $LP(n,q)$ are lines of the projective space $PG(n,q)$, $n\ge 3$. The lines of $LP(n,q)$ are
$2$-planes of $PG(n,q)$ and incidence is inclusion. Then, $LP(n,q)$ is a semipartial geometry with $s=q(q+1)$, $t=\frac{q^{n-1}-1}{q-1}-1$,
$\alpha=q+1$, and $\mu=(q+1)^2$. It is a partial geometry if and only if $n=3$. Moreover,
$v=b$ holds if and only if $n=4$. Thus, $LP(4,q)$ is a $(v_k;\lambda,\mu)$ configuration with
\begin{equation}\label{LPpar}
\begin{array}{c}
v=(q^2+1)(q^4+q^3+q^2+q+1), \kern 2mm k=q^2+q+1,\\[2mm]
\lambda=q^3+2q^2+q-1, \kern 4mm \mu=(q+1)^2.
\end{array}
\end{equation}
It is self-dual and has full automorphism group $P\Gamma L(5,q)$.

We now describe transformations of $LP(4,q)$ into strongly regular configurations that are not semipartial
geometries. We refer to them as \emph{polarity transformations}; they are similar to the construction of
polarity designs in~\cite{JT09}. Let~$H_0$ be a hyper\-plane
of $PG(4,q)$. As a subgeometry, $H_0$ is isomorphic to $PG(3,q)$ and admits a polarity~$\pi$, i.e.\ an inclusion-reversing
involution. The polarity permutes the set of projective lines contained in $H_0$ and exchanges the set of points in~$H_0$ with
the set of planes in~$H_0$. We modify incidence
of the elements of $LP(4,q)$ contained in~$H_0$: a point~$L$ (projective line contained in $H_0$) is
incident with a line $p$ (projective plane contained in $H_0$) if $\pi(L)\subseteq p$. For the remaining pairs
$(L,p)$, with $L$ or $p$ not contained in $H_0$, incidence remains unaltered. We claim that the
new incidence structure $LP(4,q)^\pi$ is a $(v_k;\lambda,\mu)$ configuration with parameters~\eqref{LPpar}.

The point and line degrees clearly remain the same and there is at most one line through every pair of
points. The point graphs of $LP(4,q)^\pi$ and $LP(4,q)$ are identical. This follows from the next lemma.

\begin{lemma}\label{coplanarity}
Two projective lines of $PG(n,q)$ are coplanar if and only if they intersect.
\end{lemma}

If $L_1$ and $L_2$ are projective lines of~$H_0$, then $\pi(L_1)$, $\pi(L_1)$ are contained in a
plane~$p$ if and only if $L_1$, $L_2$ intersect in the point $\pi(p)$ and hence, by Lemma~\ref{coplanarity},
are contained in some plane~$p'$. The line graph of $LP(4,q)^\pi$ is changed, but remains strongly regular because
of Theorem~\ref{linegraph}.

To see that the new configuration $LP(4,q)^\pi$ is not a semipartial geometry, take a plane $p$ in $H_0$
and a projective line $L$ intersecting~$H_0$ in the point $\pi(p)$. Then, $(L,p)$ is
a non-incident point-line pair of $LP(4,q)^\pi$. If $\pi(M)\subseteq p$, then $M$ contains $\pi(p)$ and is
coplanar with~$L$, i.e.\ collinear as a point of the configuration. Hence, all $q^2+q+1$ points on $p$ are
collinear with $L$, whereas in a semipartial geometry the number is always $0$ or $\alpha=q+1$. The configurations
$LP(4,q)$ and $LP(4,q)^\pi$ are therefore not isomorphic. Configurations obtained by transforming $LP(4,q)$
with different polarities are all isomorphic, because the composition of two polarities is an isomorphism.

We define a dual transformation of $LP(4,q)$ in the following manner. Take a point $P_0$ of $PG(4,q)$ and
consider the quotient geometry of lines, planes and solids containing $P_0$. It is isomorphic to $PG(3,q)$
and admits a polarity~$\pi'$ permuting the planes through~$P_0$ and exchanging the lines and solids
through~$P_0$. We modify incidence in $LP(4,q)$ for projective lines $L$ and planes $p$ through $P_0$:
they are incident if $L\subseteq \pi'(p)$. The new configuration $LP(4,q)_{\pi'}$ is isomorphic to the
dual of $LP(4,q)^\pi$ and therefore strongly regular with
parameters~\eqref{LPpar}, but not a semipartial geometry. The line graphs of $LP(4,q)_{\pi'}$ and $LP(4,q)$
are identical, while the point graph of $LP(4,q)_{\pi'}$ is changed.

A fourth $(v_k;\lambda,\mu)$ configuration is obtained if we take a non-incident point-hyperplane pair
$P_0$, $H_0$ of $PG(4,q)$ and apply both transformations. The lines and planes in $H_0$ are different from
the lines and planes through~$P_0$, so incidence is changed in disjoint parts of the configuration. The resulting
configuration $LP(4,q)_{\pi'}^\pi$ has the same line graph as $LP(4,q)^\pi$ and the same point graph as
$LP(4,q)_{\pi'}$ and is self-dual. This proves the following theorem.

\begin{theorem}\label{LPtransformed}
For every prime power~$q$, there are at least four strongly regular $(v_k;\lambda,\mu)$ configuration
with parameters~\eqref{LPpar}. One of them is the semipartial geometry $LP(4,q)$ and the others are
not semipartial geometries.
\end{theorem}

We now present an infinite family of strongly regular configurations with parameters
different from semipartial geometries. The construction works by deleting a
suitable subset from a projective plane, similarly as constructions of elliptic
semiplanes.

\begin{theorem}\label{triangle}
Let $\P$ be a projective plane of order $n\ge 5$ and $A$, $B$, $C$
be three non-collinear points. By deleting all points on the lines
$AB$, $AC$, $BC$ and all lines through the points $A$, $B$, $C$, there
remains a strongly regular $(v_k;\lambda,\mu)$ configuration with $v=(n-1)^2$, $k=n-2$,
$\lambda=(n-4)^2+1$, and $\mu=(n-3)(n-4)$. This configuration is not an
$(\alpha,\beta)$-geometry.
\end{theorem}

\begin{proof}
The number of points and lines in the remaining configuration is
$v=n^2+n+1-3-3(n-1)=(n-1)^2$ and they are of degree $k=n-2$. Let
$P$ and $Q$ be two remaining points that are collinear, i.e.\
are not on a line of $\P$ through $A$, $B$ or $C$. Then the points
non-collinear with $P$ are the remaining points on the
lines $AP$, $BP$, $CP$. There are $3(n-2)$ such points, and as many
for $Q$. The points non-collinear with both $P$ and $Q$ are
the intersections of one of the lines $AP$, $BP$, $CP$ with one of
the lines $AQ$, $BQ$, $CQ$; there are $6$ such points. By
inclusion-exclusion, the number of points in the remaining
configuration collinear with both $P$ and $Q$ is $\lambda=(n-1)^2-2
-6(n-2)+6=(n-4)^2+1$. If the points~$P$ and~$Q$
are non-collinear, a similar count shows that the number of
points in the remaining configuration collinear with both $P$ and $Q$ is
$\mu=(n-3)(n-4)$.

Let $(P,\ell)$ be a non-incident point-line pair of the remaining
configuration. We now count the points on $\ell$ collinear with $P$.
Let $A'$, $B'$, $C'$ be the intersections of $BC$, $AC$, $AB$ with~$\ell$.
These are the deleted points of~$\ell$. If the lines $AA'$, $BB'$, $CC'$
are concurrent, $P$ lies on $0$, $1$ or $3$ of these lines. Then there are
$n-5$, $n-4$ or $n-2$ points on~$\ell$ collinear with~$P$. In this case, the points
$A$, $B$, $C$, $A'$, $B'$, $C'$ and the common point of $AA'$, $BB'$, $CC'$
form a Fano subplane, so this can only occur if $n$ is even or $\P$ is non-Desarguesian.
On the other hand, if the lines $AA'$, $BB'$, $CC'$
are not concurrent, $P$ lies on $0$, $1$ or $2$ of these lines and there are
$n-5$, $n-4$ or $n-3$ points on~$\ell$ collinear with~$P$. In both cases there are three
possibilities for the number of points on~$\ell$ collinear with $P$, so the configuration
is not an $(\alpha,\beta)$-geometry.
\end{proof}

The associated graphs have parameters
$$SRG((n-1)^2,(n-2)(n-3),(n-4)^2+1,(n-3)(n-4)).$$
These are pseudo Latin square graphs $LS_{n-3}(n-1)$, see \cite[Section 8.4.2]{BvM21}.
For $n=5$, we get the Shrikhande graph~\cite{SSS59} which is not a Latin square graph.
For $n=7$, the graphs have parameters $LS_4(6)$ and are not Latin square graphs
because there are no orthogonal Latin squares of order~$6$.

In the smallest case $n=5$, the $(16_3;2,2)$ configuration of Theorem~\ref{triangle}
can be extended to a $(16_4;8,12)$ configuration by adding a point to every line.
This is a $(4,4)$-net and can be embedded in the projective plane of order~$4$.
This is an interesting transformation of the projective plane of order~$5$ into the
projective plane of order~$4$, but it does not generalize to $n>5$.

In the Desarguesian projective plane $PG(2,q)$, all triangles $\{A,B,C\}$
are equivalent and Theorem~\ref{triangle} gives just one
strongly regular configuration up to isomorphism, being self-dual. The
smallest non-De\-sar\-guesian projective planes are of order~$9$: the Hall
plane, its dual, and the self-dual Hughes plane.  The Hall plane contains
six inequivalent triangles and as many non-isomorphic $(64_7;26,30)$ configurations
arise from Theorem~\ref{triangle}. These
configurations are not self-dual. Of course, they are duals of the configurations
derived from the dual Hall plane. The Hughes plane contains $16$ inequivalent
triangles. The corresponding configurations are not isomorphic; $10$
are self-dual and there are $3$ dual pairs. Information on the orders of full
automorphism groups of these configurations is given in Table~\ref{table64}.

\begin{table}[h]
\begin{tabular}{|c c c|c c c|}
\hline
Plane & $|\Aut|$ & \#Cf & Plane & $|\Aut|$ & \#Cf \\
\hline
\rule{0mm}{12pt}$PG(2,9)$ & 768 & 1 & Hughes & 144 & 1 \\
\cline{1-3}
\rule{0mm}{12pt}Hall & 768 & 1 & & 48 & 1 \\
\rule{0mm}{11pt} & 96 & 2 & & 32 & 1 \\
\rule{0mm}{11pt} & 12 & 2 & & 18 & 1 \\
\rule{0mm}{11pt} & 6 & 1 & & 12 & 3 \\
\cline{1-3}
\rule{0mm}{12pt}Dual Hall & 768 & 1 & & 6 & 4 \\
\rule{0mm}{11pt} & 96 & 2 & & 4 & 3 \\
\rule{0mm}{11pt} & 12 & 2 & & 2 & 1 \\
\rule{0mm}{11pt} & 6 & 1 & & 1 & 1 \\
\hline
\end{tabular}
\vskip 3mm \caption{Distribution of $(64_7;26,30)$ configurations % from Theorem~\ref{triangle}
by order of full automorphism group.}\label{table64}
\end{table}

Configurations obtained from different projective planes of order $9$ are not isomorphic.
Hence, the total number of $(64_7;26,30)$ configurations arising from Theorem~\ref{triangle}
is $29$. We could not find any other examples with these parameters.
This, together with the uniqueness results of Section~\ref{sec5} (Corollary~\ref{src16}
and Proposition~\ref{src36}), seems to suggest that every strongly regular configuration
with parameters from Theorem~\ref{triangle} can be uniquely embedded in a projective plane
of order~$n$, but we do not have a proof.

\section{Strong deficient difference sets}\label{sec4}

Next we present constructions of strongly regular configurations
using difference sets. Let $G$ be a group of order~$v$. A subset
$D\subseteq G$ of size~$k$ is a \emph{deficient difference set}
if for every $x\in G\setminus\{1\}$, there is at most one pair
$(d_1,d_2)\in D\times D$ such that $x=d_1^{-1}d_2$.
Shortly, the left differences $ d_1^{-1}d_2$ must all be distinct.
This is equivalent with the right differences $d_1 d_2^{-1}$ being
distinct. The elements of~$G$ as points and the \emph{development}
$\dev D=\{gD \mid g\in G\}$ as lines form a symmetric $(v_k)$
configuration. The configuration has $G$ an automorphism group
acting regularly on the points and lines~\cite{MF08, MPW87}. In
the cyclic case $G=\Z_v$, deficient difference sets are also called
\emph{modular Golomb rulers}~\cite{BS21}.

Let $\Delta(D)=\{d_1^{-1}d_2 \mid d_1,d_2\in D, d_1\neq d_2\}$ be
the set of left differences of~$D$. This is a subset of $G\setminus
\{1\}$ of size $k(k-1)$. For a group element $x\neq 1$, denote by
$n(x)=|\Delta(D)\cap x\Delta(D)|$. Suppose that $n(x)=\lambda$ for
every $x\in \Delta(D)$, and $n(x)=\mu$ for every $x\not\in
\Delta(D)$. We shall call a subset $D$ with this property a
\emph{strong deficient difference set (SDDS)} for
$(v_k;\lambda,\mu)$.

\begin{theorem} Let $G$ be a group and
$D\subseteq G$ a strong deficient difference set for
$(v_k;\lambda,\mu)$. Then, $(G,\dev D)$ is a strongly regular $(v_k;\lambda,\mu)$
configuration with~$G$ as an automorphism group acting regularly on the points
and lines. Conversely, any strongly regular $(v_k;\lambda,\mu)$ configuration with
an automorphism group~$G$ acting regularly on the points and lines can be obtained
from a $SDDS$ in $G$.
\end{theorem}

\begin{proof}
Two points $x,y\in G$ are collinear if and only if $x^{-1}y\in
\Delta(D)$. Let us count the number of points $z\in
G\setminus\{x,y\}$ collinear with both $x$ and $y$. This is
equivalent with $x^{-1}z\in \Delta(D)$ and $y^{-1}z\in \Delta(D)$,
or $z\in x\Delta(D)\cap y\Delta(D)$, or $x^{-1}z \in \Delta(D)\cap
x^{-1}y\Delta(D)$. The number of such points $z$ is $\lambda$ if
$x^{-1}y\in \Delta(D)$, i.e.\ if $x$ and $y$ are collinear, and
$\mu$ otherwise. Hence, the point graph is strongly regular with
parameters $SRG(v,k(k-1),\lambda,\mu)$.

Conversely, assume a $(v_k;\lambda,\mu)$ configuration possesses an
automorphism group~$G$ acting regularly. Then the points can be identified
with the elements of~$G$ and every block is a deficient difference set
generating this configuration. The argument above shows that it is
in fact a $(v_k;\lambda,\mu)$ SDDS.
\end{proof}

Configurations constructed from $PG(2,q)$ by Theorem~\ref{triangle}
can be obtained from strong deficient difference sets in the group
$G=\F_q^*\times \F_q^*$. Here, $\F_q^*$ denotes the multiplicative group of
the finite field $\F_q$, isomorphic to the cyclic group $\Z_{q-1}$.
If two of the points $\{A,B,C\}$ are chosen on the ``line at infinity''
and the third point as the ``origin'' $(0,0)$, points of the configuration
can be identified with pairs $(x,y)$ with $x,y\in \F_q^*$. Lines are sets
of points satisfying equations of the form $y=ax+b$, $a,b\in\F_q^*$. Hence, e.g.\
$D=\{(x,x+1) \mid x\in \F_q^*\setminus \{-1\}\}$ is a
SDDS for $(v_k;\lambda,\mu)$ with $v=(q-1)^2$, $k=q-2$, $\lambda=(q-4)^2+1$, and $\mu=(q-3)(q-4)$.
The full automorphism group of the configuration is
$((\F_q^*\times \F_q^*):\Aut(\F_q)):S_3$, where $\Aut(\F_q)$ are the
field automorphisms, and $S_3$ corresponds to collineations of $PG(2,q)$
exchanging vertices of the triangle $\{A,B,C\}$.

The two $(64_7;26,30)$ configurations with full automorphism groups of order $768$
constructed from the Hall plane and its dual (see Table~\ref{table64})
can be obtained from SDDS's in the group $G=Q_8\times Q_8$, where $Q_8=\{\pm 1,\pm i,\pm j,\pm k\}$
is the quaternion group with usual multiplication (e.g.\
$i^2=j^2=k^2=-1$, $ij=k$). The difference set
$$D_1=\{(1,1), (i,-k), (j,k), (k,-j), (-i,j), (-j,i), (-k,-i)\}$$
gives the configuration constructed from the Hall plane and
$$D_2=\{(1,1), (i,-k), (j,j), (k,-j), (-i,-i), (-j,i), (-k,k)\}$$
gives the dual configuration. The Hall plane of order $9$ and its dual are
coordinatized by the quaternionic near-field. The first configuration
arises from Theorem~\ref{triangle} when two of the points $\{A,B,C\}$
are chosen on the translation line of the Hall plane, and the second configuration when one
of the points is the translation point of the dual Hall plane.

We performed an exhaustive computer search for strong deficient
difference sets with parameters corresponding to proper and primitive
strongly regular configurations in groups of order $v\le 200$, using
the GAP library of small groups~\cite{GAP4}.
Apart from the examples just described, we found four other examples
not corresponding to Theorem~\ref{triangle}. The
configurations constructed from these SDDS's have flag-transitive
automorphism groups. Here are their descriptions.

\begin{example}\label{sdds13}
SDDS's for $(13_3;2,3)$ exist in the cyclic group $\Z_{13}$. There is
one SDDS fixed by the multiplier $3$: $\{7,8,11\}$. The development
has full automorphism group $\Z_{13}:\Z_3$ acting flag-transitively.
\end{example}

This is the only cyclic strongly regular configuration we found. It
can be embedded in the projective plane of order $3$ by adding a
point to every line.

\begin{example}\label{sdds96}
SDDS's for $(96_5;4,4)$ exist in the groups $\Z_4 \times S_4$, $(\Z_2
\times \Z_2 \times A_4) : \Z_2$, $D_8\times A_4$ and $\Z_2 \times
\Z_2 \times S_4$. Here is one SDDS in $\Z_4\times S_4$:
$$\{(0,id), (1, (1,4)(2,3)), (1,(1,3,4,2)),
  (1, (1,4,3)), (2, (1,2,4)) \}.$$
The developments are all isomorphic and give one self-dual configuration.
The full automorphism group is $((\Z_2 \times \Z_2 \times \Z_2 \times \Z_2) : A_6) : \Z_2$ of
order $11520$ and acts flag-transitively.
\end{example}

The associated graphs have parameters $SRG(96,20,4,4)$. Many such graphs
are known, see~\cite{BKK03, GMV06}. The graph with the largest
automorphism group of order $138240$ is the point graph of the generalized
quadrangle $pg(5,3,1)$. The graph of the $(96_5;4,4)$
configuration has full automorphism group of order $11520$.
In~\cite{BKK03}, this graph is denoted by $K''$ and the configuration is
mentioned as a ``partial linear space with five points per line and five lines
on each point''.

\begin{example}\label{sdds120}
SDDS's for $(120_8;28,24)$ exist in the symmetric group $S_5$, e.g.
\begin{align*}
 \{ & id,\, (1,2,5,3,4),\, (1,3,4,2,5),\, (1,5,3,2,4),\, (1,4)(2,3,5), \\[1mm]
 & (1,4,5,2),\, (1,2,4),\, (1,2,5)\}.
\end{align*}
Up to isomorphism one self-dual strongly regular configuration arises.
The full automorphism group is isomorphic to the alternating group~$A_8$
of size $20160$ and acts flag-transitively.
\end{example}

This $(120_8;28,24)$ configuration was constructed in~\cite{BHT97}
by embedding the $pg(7,8,4)$ of~\cite{DDT80, AMC81} into a Steiner $2$-$(120,8,1)$
design. The $135$ lines of the $pg(7,8,4)$ and the $120$ lines
of the configuration cover every pair of the $120$ points exactly once
and form a design. The point graphs of the $pg(7,8,4)$ and the $(120_8;28,24)$
configuration are complementary with parameters $SRG(120,63,30,36)$ and
$SRG(120,56,28,24)$, respectively.

The $pg(7,8,4)$ is part of an infinite family constructed from
the hyperbolic quadric in $PG(4n-1,2)$~\cite{DDT80}. The family is
denoted by $PQ^+(2n-1,2)$ and has parameters $pg(2^{2n-1}-1,2^{2n-1},2^{2n-2})$.
These parameters fit a hypothetical $(v_k;\lambda,\mu)$ configuration with
$$v=2^{2n-1}(2^{2n}-1), k=2^{2n-1}, \lambda=2^{2n-2}(2^{2n-1}-1),
\mu=2^{2n-1}(2^{2n-2}-1)$$
to make a $2$-$(v,k,1)$ design, but in \cite[Theorem 2.1]{BHT97}
it was proved that this is not possible for $n>2$. Non-isomorphic partial
geometries with the same parameters were constructed in~\cite{MS97, DD00}
that could possibly be embedded in Steiner $2$-designs.

\begin{example}\label{pg42}
SDDS's for $(155_7;17,9)$ exist in the group $G=\Z_{31}:\Z_5$. Let $G$
be represented as permutations of $\Z_{31}$ generated by $f:x\mapsto
x+1 \pmod{31}$ and $g:x\mapsto 2x \pmod{31}$. Then, $\{id,f^{12}g^4,
f^{15}g, f^{18}, f^{20}g^2,$ $f^{26}g^3, f^{30}\}$ is a SDDS. One
self-dual strongly regular configuration arises, isomorphic to the
semipartial geometry $LP(4,2)$. The full automorphism group $P\Gamma L(5,2)$
is of order $9999360$ and acts flag-transitively.
\end{example}

The configurations obtained from $LP(4,2)$ by polarity transformations
cannot be constructed from SDDS because their full automorphism groups are
not transitive. The dual pair $LP(4,2)^\pi$ and $LP(4,2)_{\pi'}$ have
full automorphism groups of order $322560$ isomorphic to $(\Z_2)^4:P\Gamma L(4,2)$.
The group acts in orbits of size $35$, $120$ on the points and $15$, $140$ on the lines
of $LP(4,2)^\pi$, and vice versa for $LP(4,2)_{\pi'}$.
The self-dual configuration $LP(4,2)^{\pi}_{\pi'}$ has full automorphism group
of order $20160$ isomorphic to $P\Gamma L(4,2)$ acting in orbits of size $15$, $35$, $105$.

Our final examples of strongly regular configurations can
also not be obtained from SDDS's. They don't admit automorphism
groups acting regularly, although some have flag-transitive
automorphism groups.

\begin{example}\label{src63}
There are at least four non-isomorphic $(63_6;13,15)$
configurations. Two of them are self-dual
with full automorphism groups $PSU(3,3) : \Z_2$ of order
$12096$ acting flag-transitively. Furthermore, there is a dual pair
with full automorphism groups $(SL(2,3) : \Z_4) : \Z_2$ of
order $192$ acting in orbits of size $1$, $6$, $24$, $32$.
\end{example}

The two self-dual $(63_6;13,15)$ configurations are related to the
smallest generalized hexagon $GH(2,2)$ (see~\cite[Section 5.7]{GR01}).
This is a $(63_3)$ configuration  with point and line graphs of
girth $12$ and diameter $6$. The graphs are distance regular, but
not strongly regular. A strongly regular $(63_6;13,15)$ configuration
can be constructed similarly as a semipartial geometry
from a Moore graph: the new configuration has the same points as
$GH(2,2)$, and lines of the new configuration are sets of $6$ points
collinear with a given point of $GH(2,2)$. The point graph of this
$(63_6)$ configuration is a $SRG(63,30,13,15)$. The other self-dual
$(63_6;13,15)$ configuration is constructed in the same way from the
dual of $GH(2,2)$. We discovered the dual pair of non-transitive
$(63_6;13,15)$ configurations computationally, by prescribing
automorphism groups.

\section{A table of feasible parameters}\label{sec5}

In the final section we present a table of feasible parameters
of strongly regular configurations with $v\le 200$. A.~E.~Brouwer's
table of strongly regular graphs~\cite{AEB} contains $437$ parameter sets
$SRG(v,d,\lambda,\mu)$ with $v\le 200$. It is known that strongly regular
graphs do not exist in $62$ cases. Among the remaining $375$
cases, we look for those with $d=k(k-1)$ for some $k\ge 3$. This way
we get $64$ parameter sets $(v_k;\lambda,\mu)$.

Eleven of the $64$ parameter sets do not satisfy Theorem~\ref{cliquecond}.
Six satisfy the theorem with equality and correspond to
partial geometries $pg(2,2,1)$, $pg(3,3,1)$, $pg(6,6,4)$, $pg(5,5,2)$,
$pg(4,4,1)$, and $pg(5,5,1)$. The $pg(q,q,1)$ with $q=2,3,4,5$ are the classical generalized
quadrangles $W(q)$ and their duals, see~\cite{PT09}.
Two non-isomorphic $pg(5,5,2)$'s are known~\cite{vLS81, CST21, VK20b}, whereas
the existence of a $pg(6,6,4)$ is open. Six of the remaining $47$ parameter
sets are eliminated by Proposition~\ref{squarecond}.

\begin{table}
\begin{center}
\begin{tabular}{|c c c c l|}
\hline
No. & $(v_k;\lambda,\mu)$ & \#Cf & \#SCf & Comments\\
\hline
\rule{0mm}{4.5mm}1 & $(10_{3};3,4)$ & \textbf{2} & \textbf{2} & \\
\rule{0mm}{4.5mm}2 & $(13_{3};2,3)$ & \textbf{1} & \textbf{1} & Proposition~\ref{src13}\\
\rule{0mm}{4.5mm}3 & $(16_{3};2,2)$ & \textbf{1} & \textbf{1} & Proposition~\ref{src16}\\
\rule{0mm}{4.5mm}4 & $(25_{4};5,6)$ & \textbf{0} & \textbf{0} & Proposition~\ref{non25-4}\\
\rule{0mm}{4.5mm}5 & $(36_{5};10,12)$ & \textbf{1} & \textbf{1} &  Proposition~\ref{src36}\\
\rule{0mm}{4.5mm}6 & $(41_{5};9,10)$ & ? & ? & \\
\rule{0mm}{4.5mm}7 & $(45_{4};3,3)$ & \textbf{0} & \textbf{0} & Proposition~\ref{non45-4}\\
\rule{0mm}{4.5mm}8 & $(49_{4};5,2)$ & \textbf{0} & \textbf{0} & Corollary~\ref{norook} \\
\rule{0mm}{4.5mm}9 & $(49_{6};17,20)$ & 1 & 1 & Theorem~\ref{triangle}\\
\rule{0mm}{4.5mm}10 & $(50_{7};35,36)$ & 211 & 111 & Proposition~\ref{src50}\\
\rule{0mm}{4.5mm}11 & $(61_{6};14,15)$ & ? & ? & \\
\rule{0mm}{4.5mm}12 & $(63_{6};13,15)$ & 4 & 2 & Example~\ref{src63}\\
\rule{0mm}{4.5mm}13 & $(64_{7};26,30)$ & 29 & 11 & Theorem~\ref{triangle}\\
\rule{0mm}{4.5mm}14 & $(81_{8};37,42)$ & ? & ? & \\
\rule{0mm}{4.5mm}15 & $(85_{6};11,10)$ & ? & ? & \\
\rule{0mm}{4.5mm}16 & $(85_{7};20,21)$ & ? & ? & \\
\rule{0mm}{4.5mm}17 & $(96_{5};4,4)$ & 1 & 1 & Example~\ref{sdds96}\\
\rule{0mm}{4.5mm}18 & $(99_{7};21,15)$ & ? & ? & \\
\rule{0mm}{4.5mm}19 & $(100_{9};50,56)$ & 1 & 1 & Theorem~\ref{triangle}\\
\rule{0mm}{4.5mm}20 & $(105_{9};51,45)$ & ? & ? & \\
\rule{0mm}{4.5mm}21 & $(113_{8};27,28)$ & ? & ? & \\
\rule{0mm}{4.5mm}22 & $(120_{8};28,24)$ & 1 & 1 & Example~\ref{sdds120}\\
\rule{0mm}{4.5mm}23 & $(121_{5};9,2)$ & \textbf{0} & \textbf{0} & Corollary~\ref{norook} \\
\rule{0mm}{4.5mm}24 & $(121_{6};11,6)$ & ? & ? & \\
\rule{0mm}{4.5mm}25 & $(121_{9};43,42)$ & ? & ? & \\
\rule{0mm}{4.5mm}26 & $(121_{10};65,72)$ & ? & ? & \\
\rule{0mm}{4.5mm}27 & $(125_{9};45,36)$ & ? & ? & \\
\rule{0mm}{4.5mm}28 & $(136_{6};15,4)$ & ? & ? & \\
\rule{0mm}{4.5mm}29 & $(136_{9};36,40)$ & ? & ? & \\
\rule{0mm}{4.5mm}30 & $(144_{11};82,90)$ & 1 & 1 & Theorem~\ref{triangle}\\
\hline
\end{tabular}
\end{center}
\vskip 3mm \caption{Feasible parameters of proper primitive strongly regular
configurations.}\label{fptable}
\end{table}

\addtocounter{table}{-1}

\begin{table}
\begin{center}
\begin{tabular}{|c c c c l|}
\hline
No. & $(v_k;\lambda,\mu)$ & \#Cf & \#SCf & Comments\\
\hline
\rule{0mm}{4.5mm}31 & $(145_{9};35,36)$ & ? & ? & \\
\rule{0mm}{4.5mm}32 & $(153_{8};19,21)$ & ? & ? & \\
\rule{0mm}{4.5mm}33 & $(155_{7};17,9)$ & 4 & 2 & Theorem~\ref{LPtransformed}\\
\rule{0mm}{4.5mm}34 & $(169_{9};31,30)$ & ? & ? & \\
\rule{0mm}{4.5mm}35 & $(169_{12};101,110)$ & ? & ? & \\
\rule{0mm}{4.5mm}36 & $(171_{11};73,66)$ & ? & ? & \\
\rule{0mm}{4.5mm}37 & $(175_{6};5,5)$ & ? & ? & \\
\rule{0mm}{4.5mm}38 & $(181_{10};44,45)$ & ? & ? & \\
\rule{0mm}{4.5mm}39 & $(196_{10};40,42)$ & ? & ? & \\
\rule{0mm}{4.5mm}40 & $(196_{13};122,132)$ & ? & ? & \\
\rule{0mm}{4.5mm}41 & $(196_{13};125,120)$ & ? & ? & \\
\hline
\end{tabular}
\end{center}
\vskip 3mm \caption{Feasible parameters of proper primitive strongly regular
configurations (continued).}
\end{table}

Thus, there are $41$ feasible parameter sets $(v_k;\lambda,\mu)$ of
proper and primitive strongly regular configurations with $v\le 200$.
The para\-meters are listed in Table~\ref{fptable} along with information
on the numbers of strongly regular configurations (\#Cf)
and self-dual strongly regular configurations (\#SCf) up to isomorphism.
A number in \textbf{boldface} indicates that this is the exact number,
otherwise it is a lower bound.

In the smallest case $(10_3;3,4)$, there are altogether ten combinatorial $(10_3)$
configurations denoted by $(10_3)_i$, $i=1,\ldots,10$ in \cite[Section~2.2]{BG09}.
Two of them are strongly regular: the Desargues configuration~$(10_3)_1$
and the configuration $(10_3)_4$ depicted in Figure~\ref{smallest}. Interestingly, $(10_3)_4$
is the only one of the ten $(10_3)$ configurations that cannot be drawn with
straight lines, i.e.\ that is not a \emph{geometric configuration} (see~\cite{BG09, PS13}).
In the next two cases $(13_3;2,3)$ and $(16_3;2,2)$, the total numbers of
$(13_3)$ and $(16_3)$ configurations are also known: $2036$~\cite{HG90} and $3\,004\,881$~\cite{BBP00},
respectively. Since the number of combinatorial $(v_k)$ configurations grows
rapidly with~$v$, a better approach to classifying strongly regular configurations is
through the associated graphs.

Suppose that a strongly regular graph $\Gamma$ with parameters $SRG(v,k(k-1),\lambda,\mu)$ is the point graph of
a $(v_k;\lambda,\mu)$ configuration. Every line of the configuration gives
a clique of size~$k$ in~$\Gamma$. Thus, there must be~$v$ such cliques with every pair of
them intersecting in at most one point. Given the graph~$\Gamma$, we define
the \emph{clique graph} $\C(\Gamma)$ with vertices being $k$-cliques in~$\Gamma$. Vertices
of $\C(\Gamma)$ are adjacent if the cliques intersect in at most one point. The task
is to find the cliques of size $v$ in $\C(\Gamma)$.

Up to isomorphism, there is a unique graph $SRG(13,6,2,3)$, the
Paley graph. The vertices of $\Gamma$ are elements of the finite field $\F_{13}$
with two vertices being adjacent if their difference is a quadratic residue. Using
Cliquer~\cite{NO03}, we found $26$ cliques of size $3$ in $\Gamma$.
The clique graph $\C(\Gamma)$ has $26$ vertices and $286$ edges. Using Cliquer once more,
we found exactly two cliques of size $13$ in $\C(\Gamma)$, corresponding
to isomorphic $(13_3;2,3)$ configurations. This proves that the cyclic configuration
constructed in Example~\ref{sdds13} is unique.

\begin{proposition}\label{src13}
There is one $(13_3;2,3)$ configuration up to isomorphism.
\end{proposition}

There are two graphs with parameters $SRG(16,6,2,2)$. One of them is the
Shrikhande graph~\cite{SSS59} with full automorphism group of order $192$.
Similarly as for the previous parameters, we found $32$ cliques of size~$3$ in~$\Gamma$
and two cliques of size~$16$ in~$\C(\Gamma)$, corresponding
to isomorphic $(16_3;2,2)$ configurations.

The other $SRG(16,6,2,2)$ has full automorphism group of order $1152$. This is
the $4\times 4$ \emph{rook graph}, sometimes also called the \emph{lattice graph}
or \emph{grid graph}. Vertices of the $n\times n$ rook
graph $R_n$ are pairs $(x,y)$ with $x,y\in\{1,\ldots,n\}$.
Two vertices $(x_1,y_1)$, $(x_2,y_2)$ are adjacent if $x_1=x_2$ or $y_1=y_2$ holds.
The graph $R_n$ is strongly regular with parameters $SRG(n^2,2(n-1),n-2,2)$ and has
$2n$ maximal cliques of size $n$, being sets of vertices with a fixed coordinate.
Any clique of size at least $2$ is contained in exactly one of these maximal cliques.
If $R_n$ is the point graph of a $(v_k;\lambda,\mu)$ configuration, then $2(n-1)=k(k-1)$
holds. This is equivalent with $n={k\choose 2}+1$ and the configuration
would have parameters
\begin{equation}\label{rookparam}
v=\left( {k\choose 2}+1\right)^2,\kern 3mm \lambda={k\choose 2}-1,\kern 3mm \mu=2.
\end{equation}
We now prove that this cannot occur.

\begin{theorem}\label{thmrook}
The $n\times n$ rook graph is not the point graph of a strongly regular configuration.
\end{theorem}

\begin{proof}
Lines of the configuration would give a set $\C$ of $v=n^2$ cliques of size $k$
in $R_n$, pairwise intersecting in at most one vertex. A maximal clique of size
$n={k\choose 2}+1$ contains no more than $\frac{n(n-1)}{k(k-1)}$ of the cliques
in~$\C$, because each of the $n(n-1)$ pairs of
distinct vertices is contained in at most one $k$-clique, and a $k$-clique covers
$k(k-1)$ pairs. Therefore, $\C$ is not larger than $2n\cdot \frac{n(n-1)}{k(k-1)}$.
This is equal to $v=n^2$, and therefore the cliques of $\C$ contained in a given
$n$-clique cover every pair of its $n$ vertices exactly once. In this way we get
a Steiner $2$-$(n,k,1)$ design. If $r=\frac{n-1}{k-1}$ is the replication number
of the design, Fisher's inequality $r\ge k$ gives $n-1\ge k(k-1)$, a
contradiction with $n={k\choose 2}+1$.
\end{proof}

Together with the discussion about the Shrikhande graph, this proves that
the strongly regular configuration constructed from $PG(2,5)$ by Theorem~\ref{triangle}
is unique.

\begin{corollary}\label{src16}
There is one $(16_3;2,2)$ configuration up to isomorphism.
\end{corollary}

Furthermore, Theorem~\ref{thmrook} eliminates infinitely many feasible parameter sets of
strongly regular configurations.

\begin{corollary}\label{norook}
Strongly regular configurations with parameters~\eqref{rookparam} do not
exist for $k>3$.
\end{corollary}

\begin{proof}
In~\cite{SSS59}, Shrikhande proved that for $n>4$ the only strongly regular
graph with parameters $SRG(n^2,2(n-1),n-2,2)$ is the $n\times n$ rook graph.
\end{proof}

We can eliminate two more parameter sets $(v_k;\lambda,\mu)$ and prove uniqueness for another
computationally, when all $SRG(v,k(k-1),\lambda,\mu)$ graphs are known.

\begin{proposition}\label{non25-4}
Strongly regular $(25_4;5,6)$ configurations do not exist.
\end{proposition}

\begin{proof}
Up to isomorphism, there are exactly $15$ strongly regular graphs with parameters
$SRG(25,12,5,6)$~\cite{AP73, MR73}. The adjacency matrices are available on
E.~Spence's web page~\cite{ES}. Cliquer found from $73$ to $90$ cliques of
size $4$ in these graphs, but none of the corresponding clique graphs
$\C(\Gamma)$ contain a clique of size $25$.
\end{proof}

\begin{proposition}\label{src36}
There is one $(36_5;10,12)$ configuration up to isomorphism.
\end{proposition}

\begin{proof}
There are exactly $32\,548$ graphs $SRG(36,20,10,12)$~\cite{MS01}. Adjacency
matrices of the complementary graphs are available on the web page~\cite{ES}.
Using Cliquer, we found that the $SRG(36,20,10,12)$ graphs~$\Gamma$ contain
from $132$ to $336$ cliques of size $5$. Only one of the corresponding clique
graphs $\C(\Gamma)$ contains a clique of size~$36$. This happens when the
complementary graph $\overline{\Gamma}$ with parameters $SRG(36,15,6,6)$ is
the graph constructed from the cyclic Latin square of order~$6$. Two
strongly regular configurations arise, both isomorphic to the configuration
constructed from $PG(2,7)$ by Theorem~\ref{triangle}.
\end{proof}

\begin{proposition}\label{non45-4}
Strongly regular $(45_4;3,3)$ configurations do not exist.
\end{proposition}

\begin{proof}
There are $78$ graphs $SRG(45,12,3,3)$~\cite{CDS06}.
Adjacency matrices are available on~\cite{ES}. The graphs contain
from $12$ to $135$ cliques of size~$4$ and the corresponding clique
graphs do not contain cliques of size $45$.
\end{proof}

It is also known that graphs with parameters $SRG(50,42,35,36)$ are unique, i.e.\
isomorphic to the complement of the Hoffman-Singleton graph.
This graph has $2\,708\,150$ cliques of size $7$ and we could not classify
all cliques of size $50$ in $\C(\Gamma)$. There may be other $(50_7;35,36)$
configurations apart from the $211$ examples of Proposition~\ref{src50}.

\end{document}